\documentclass{amsart}[12pt]

\usepackage[hidelinks]{hyperref}

\usepackage{yfonts} %
\usepackage{amssymb} %
\usepackage{amsthm}
\usepackage{array}
\usepackage{booktabs}%
\usepackage{ctable} %
\usepackage{hhline}%
\usepackage{xy} %
\usepackage{epsfig}%
\usepackage{color}%
\usepackage{upgreek}
\usepackage[english]{babel}
\usepackage{epigraph}%
\usepackage{fancybox}%
\usepackage{shadow}
\usepackage{afterpage}
\usepackage{mathrsfs}
\usepackage{enumitem}
\usepackage{subcaption}
\usepackage{graphicx}
\usepackage{type1cm}
\usepackage{eso-pic}
\usepackage{color}
\usepackage[foot]{amsaddr}
\usepackage{verbatim}

\newtheorem{theorem}{Theorem}
\newtheorem{lemma}{Lemma}
\newtheorem{proposition}{Proposition}
\theoremstyle{definition}
\newtheorem{definition}{Definition}
\newtheorem{remark}{Remark}

\theoremstyle{plain}

\newcommand{\otoprule}{\midrule[\heavyrulewidth]}
\newcommand{\vt}{\vspace{.1cm}}
\newcommand{\vtt}{\vspace{.2cm}}
\newcommand{\R}{\mathbb{R} }
\newcommand{\Q}{\mathbb{Q} }
\newcommand{\q}{\mathbb{Q}_{\epsilon}^n }
\newcommand{\Qe}{\mathbb{Q}_\epsilon^n}

\newcommand{\h}{\mathbb{H} }

\newcommand{\s}{\mathbb{S}}

\renewcommand{\rho}{\varrho}
\renewcommand{\theta}{\varTheta}
\renewcommand{\Theta}{\varTheta}
\renewcommand{\Lambda}{\varLambda}
\renewcommand{\Sigma}{\varSigma}
\renewcommand{\tau}{\uptau}
\newcommand{\wi}{I\times_\omega\Q_\epsilon^n}

\usepackage{amsmath}

\newcommand{\overbar}[1]{\mkern 1.5mu\overline{\mkern-1.5mu#1\mkern-1.5mu}\mkern 1.5mu}

\renewcommand{\theta}{\Theta}
\newcommand{\transv}{\mathrel{\text{\tpitchfork}}}
\makeatletter
\newcommand{\tpitchfork}{%
  \vbox{
    \baselineskip\z@skip
    \lineskip-.52ex
    \lineskiplimit\maxdimen
    \m@th
    \ialign{##\crcr\hidewidth\smash{$-$}\hidewidth\crcr$\pitchfork$\crcr}
  }%
}
\makeatother

\begin{document}

\title[Einstein Hypersurfaces]
{Einstein Hypersurfaces  of  Warped \\ Product Spaces}
\author{R. F. de Lima, F.  Manfio \and J. P. dos Santos}
\address[A1]{Departamento de Matem\'atica - Universidade Federal do Rio Grande do Norte}
\email{ronaldo.freire@ufrn.br}
\address[A2]{ICMC–Universidade de São Paulo}
\email{manfio@icmc.usp.br}
\address[A3]{Departamento de Matem\'atica - Universidade de Brasília}
\email{joaopsantos@unb.br}

\maketitle

\begin{abstract}
We consider Einstein hypersurfaces
of warped products $I\times_\omega\mathbb Q_\epsilon^n,$ where $I\subset\mathbb R$ is an open interval and
$\mathbb Q_\epsilon^n$ is the simply connected space form of dimension
$n\ge 2$ and constant sectional curvature $\epsilon\in\{-1,0,1\}.$
We show that, for all  $c\in\mathbb R$ (resp. $c>0$),  there exist rotational
hypersurfaces of constant sectional curvature $c$ in $I\times_\omega\mathbb H^n$ and
$I\times_\omega\mathbb R^n$
(resp. $I\times_\omega\mathbb S^n$), provided that $\omega$ is nonconstant.
We also show that the gradient $T$ of the  height function of any Einstein hypersurface
of $\wi$ (if nonzero)  is one of its principal directions. Then,
we consider a particular type of Einstein hypersurface
of $I\times_\omega\mathbb Q_\epsilon^n$ with non vanishing $T$ --- which we call ideal ---
and prove that, for $n>3,$ such a hypersurface $\Sigma$ has either precisely two or
precisely three distinct principal curvatures everywhere.
We show that, in the latter case, there exist such a $\Sigma$ for certain warping
functions $\omega,$ whereas in the  former case $\Sigma$
is necessarily of constant sectional curvature and rotational, regardless the warping function $\omega.$
We also characterize ideal Einstein hypersurfaces of $I\times_\omega\mathbb Q_\epsilon^n$ with
no vanishing angle function as local graphs on families of
isoparametric hypersurfaces of $\mathbb Q_\epsilon^n.$
\end{abstract}

\section{Introduction}
In this paper, we consider Einstein  hypersurfaces
of warped products $\wi,$ where $I\subset\R$ is an open interval, and
$\Qe$ stands for the simply connected space form of dimension
$n\ge 2$ and constant sectional curvature $\epsilon\in\{-1,0,1\}.$
Recall that a Riemannian manifold $(M,g)$ is called \emph{Einstein} if
$${\rm Ric}_M=\Lambda g,$$ where ${\rm Ric}_M$ is the Ricci tensor of $(M,g)$ and
$\Lambda$ is a constant. More specifically, we call such an $M$ a
$\Lambda$-Einstein manifold.

Riemannian manifolds of constant sectional curvature (CSC, for short)
are the simplest  examples of Einstein manifolds, which we  call \emph{trivial}.
It was shown by P. Ryan \cite{ryan} that,
for $\epsilon\le 0,$ any Einstein hypersurface
of $\mathbb Q_\epsilon^{n+1}$ is trivial, and also that there
exist nontrivial Einstein hypersurfaces in $\s^{n+1}.$
Inspired by this result, the main question we address here is the following:
\[
\text{\emph{Under which conditions an Einstein hypersurface of $\wi$ is necessarily trivial}?}
\]

For $\omega$  constant and $\epsilon\ne 0,$ it was proved in \cite{lps} that any Einstein hypersurface of
$\R\times_\omega\Qe$   is trivial.
On the other hand, this is not true for certain nonconstant warping functions $\omega.$
Indeed, writing $\s^2(1/\sqrt 2)$ for the $2$-sphere of $\R^3$ of radius $1/\sqrt 2,$
it is known that the product $\s^2(1/\sqrt 2)\times\s^2(1/\sqrt 2)$
--- endowed with its canonical Riemannian metric ---
is an Einstein manifold of nonconstant sectional curvature which
is naturally embedded in $\s^{5}-\{p,-p\}=(0,\pi)\times_{\sin t}\s^4,$
where $p$ is a suitable point of the unit sphere $\s^{5}\subset\R^6.$

It should be noticed that, in answering the above question, one may
conclude that certain manifolds cannot be isometrically immersed
in $\wi.$
For instance, from the aforementioned result in \cite{lps},
there is no isometric immersion
of the product
$\s^2(1/\sqrt 2)\times\s^2(1/\sqrt 2)$  into  $\R\times \Q_\epsilon^4,$
$\epsilon\ne 0.$

A major property of Einstein hypersurfaces of $\wi$
we establish here is that they have the gradient $T$
of their height functions (when nonzero)  as principal directions.
Hypersurfaces satisfying this condition, which we call the $T$-\emph{property}, have been considered in
several works on  constant curvature (mean and sectional)
hypersurfaces of  product spaces $\R\times M$, $M$ a Riemannian manifold
(see, e.g., \cite{dillenetal, delima-roitman, lps,
manfio-tojeiro, manfio-tojeiro-veken, tojeiro}).
We add that, in product spaces $\wi,$ the class of hypersurfaces
with the $T$-property is abundant and includes all rotational ones.

Considering our purposes here,  a primary concern
should be to prove existence of trivial Einstein hypersurfaces in
$\wi.$ This is accomplished in our first result, as stated below.
It should be mentioned that  we conceive rotational hypersurfaces
as in \cite{dillenetal}, dividing them into three types:
spherical, parabolic, and hyperbolic (see Section \ref{sec-csc} for details).

\begin{theorem} \label{th-existenceCSC01}
Given an open interval $I\subset\R,$ let $\omega:I\rightarrow \omega(I)$ be a diffeomorphism
such that $|\omega'|>1$ on $I.$ Then, the following assertions
hold:
\begin{itemize}[parsep=1ex]
  \item[\rm i)] For all $c\in\R,$
  there  exists a rotational spherical-type hypersurface
of constant sectional curvature $c$  in $I\times_\omega\q,$ where $\epsilon\in\{0,-1\}$.
  \item[\rm ii)]  For all $c>0,$ there  exists a rotational spherical-type hypersurface
of constant sectional curvature $c$  in $I\times_\omega\s^n.$
\item[\rm iii)]  For all $c<0,$ there exist in $I\times_\omega\h^n$
a rotational parabolic-type hypersurface and a rotational
hyperbolic-type hypersurface which are both of constant sectional curvature $c.$
  \item[\rm iv)]  There exists a flat rotational parabolic-type  hypersurface   in
$I\times_\omega\h^n.$
\end{itemize}
\end{theorem}

We remark that  Theorem \ref{th-existenceCSC01}
is of local nature. So,
there is no loss of generality in assuming the (nonconstant)
warping function $\omega$ to be
a diffeomorphism over its image.
By the same token,
we can  assume that, on $I,$
$|\omega'|$ is bounded below by a positive constant.
Hence, up to scaling, one has $|\omega'|>1.$
Considering these facts, Theorem \ref{th-existenceCSC01}  shows that,
for virtually any nonconstant warping function $\omega,$
the class of local rotational CSC hypersurfaces of $\wi$ is abundant.

To each warped product $\wi,$ we  associate  two real functions $\alpha$ and $\beta$
defined on $I,$ which arise from the expression of the curvature tensor of $\wi$ in terms of the curvature
tensor of $\Qe$ and the unit  tangent $\partial_t$ to the horizontal factor $I\subset\R.$
Then, assuming $n>3,$ we   show  that
a CSC hypersurface $\Sigma$ of $\wi$ is essentially rotational.
More precisely, we have the following result
($\xi$ always denotes the height function of the hypersurface).

\begin{theorem} \label{th-CSC}
Let $\Sigma\subset I\times_\omega \Q_\epsilon^n$ $(n>3)$ be a hypersurface with constant sectional
curvature  on which  $(\beta\circ\xi)T$ never vanishes.
Then, $\Sigma$ is rotational.
\end{theorem}

Next, we turn our attention to Einstein hypersurfaces of $\wi$
focusing on our main question.
Firstly, as stated in the next two theorems, we characterize  them
according to whether  $(\beta\circ\xi)T$ vanishes everywhere or nowhere on
the hypersurface.

\begin{theorem} \label{th-betazero}
For $n>3,$ let $\Sigma$ be a connected $\Lambda$-Einstein hypersurface of $\wi$ such
that $(\beta\circ\xi)T$ vanishes on $\Sigma.$
Then,  $\alpha\circ\xi$ is  constant. Furthermore, defining
\[
\sigma:=\Lambda+(n-1)(\alpha\circ\xi),
\]
the following assertions hold:
\begin{itemize}[parsep=1ex]
  \item[\rm i)] If $\sigma>0,$  $\Sigma$ is trivial and totally umbilical (with constant umbilical function).
   \item[\rm ii)] If $\sigma=0,$  $\Sigma$ is trivial and its shape operator has rank at most one.
  \item[\rm iii)]If $\sigma<0,$  $\Sigma$ is nontrivial  and has  precisely
  two distinct nonzero constant principal curvatures with opposite signs, both with constant
  multiplicity  $\ge 2.$ In particular, $\Sigma$ has constant mean curvature.
\end{itemize}
\end{theorem}

Connected Einstein hypersurfaces on which $(\beta\circ\xi)T$ never vanishes
have many interesting properties. For this reason, we shall call them \emph{ideal.}
The following result supports  our claim.

\begin{theorem}  \label{th-einstein}
For $n>3,$ let $\Sigma$ be an ideal Einstein hypersurface of $\wi.$
Then, $\Sigma$ has the $T$-property. In addition,
the following assertions are equivalent:
\begin{itemize}[parsep=1ex]
  \item[\rm i)] $\Sigma$ is trivial.
  \item[\rm ii)] Any vertical section $\Sigma_t=\Sigma\transv(\{t\}\times\Q_\epsilon^n)$ is
  totally umbilical.
  \item[\rm iii)] $\Sigma$ is rotational.
\end{itemize}
\end{theorem}

The second fundamental property of ideal Einstein hypersurfaces
of $\wi$ we establish here is that they are
local horizontal graphs of  functions $\phi=\phi(s)$
whose level hypersurfaces $f_s$ are parallel and isoparametric (i.e, have constant
principal curvatures) in $\q.$ We call such a hypersurface a $(\phi,f_s)$-\emph{graph}.

\begin{theorem}  \label{L1-grapheinstein}
For $n>3,$ let $\Sigma$  be an  oriented ideal Einstein hypersurface of $\wi$ with unit normal
$N$  and non vanishing angle function $\theta:=\langle N,\partial_t\rangle.$
Under these conditions, the following hold:
\begin{itemize}[parsep=1ex]
  \item [\rm i)] $\Sigma$ is locally a $(\phi,f_s)$-graph in $\wi.$
  \item [\rm ii)] The parallel family $\{f_s\}$ of any local $(\phi,f_s)$-graph of $\Sigma$
   is isoparametric, and each $f_s$ has at most two distinct principal curvatures.
  \item [\rm iii)] $\Sigma$ is trivial, provided that the principal curvature associated to
  $T$ vanishes everywhere on $\Sigma.$
\end{itemize}
\end{theorem}

Theorem \ref{th-betazero}-(iii) opens the possibility of existence of
nontrivial Einstein  hypersurfaces in $\wi.$
As we have seen, they are easily obtained in
$\mathbb{S}^{n+1},$ which turns out to be
a space form that can be expressed as a warped product.
However, nontrivial  examples
of Einstein hypersurfaces of general warped products $\wi$
are not easy to  find. Here, we consider this problem restricting ourselves to the class
of constant angle hypersurfaces of $\wi,$ obtaining the following result.

\begin{theorem}  \label{Einstein-cylinder}
For $n>3,$ let $\Sigma$  be a connected oriented $\Lambda$-Einstein
hypersurface (not necessarily ideal) of $\wi$ with constant
angle function $\theta\in[0,1).$
Then, the following assertions hold:
\begin{itemize}[parsep=1ex]
  \item [\rm i)] If $\theta\ne 0,$ $\Sigma$ is trivial.
  \item [\rm ii)] If $\theta=0,$ $\omega$ is necessarily a solution of the ODE:
  \[
  (\omega')^2 + \dfrac{\Lambda}{n-1} \omega^2 + c =0, \,\, c\in\R,
  \]
  on an open subinterval of \,$I,$ and $I\times_\omega\q$ has
  nonconstant sectional curvature if $c+\epsilon\ne0.$  Regarding $\Sigma,$
  we distinguish the following cases, according to the sign of $c+\epsilon:$
  \begin{itemize}[parsep=1ex]
  \item[\rm a)] $c+\epsilon=0:$ $\Sigma$ is trivial and its shape operator has rank at most one.
  \item[\rm b)] $c+\epsilon<0:$ $\Sigma$ is  trivial and constitutes a cylinder over
  a totally umbilical hypersurface of \,$\mathbb{Q}^n_\epsilon.$
  \item[\rm c)] $c +\epsilon>0:$  $\Sigma$ is  nontrivial and constitutes a
  cylinder over a product
  of two spheres of \,$\s^n$  (so that $\epsilon$ is necessarily $1$ in this case).
  \end{itemize}
  \end{itemize}
\end{theorem}

We point out that Theorem \ref{Einstein-cylinder}-(ii)(c)
provides  examples of  nontrivial Einstein hypersurfaces in
warped products $I\times_\omega\s^n$ which are not  space forms.

Putting together Theorems \ref{th-einstein}--\ref{Einstein-cylinder} with classical
results on classification of isoparametric hypersurfaces in space forms,
we obtain the following characterization of ideal Einstein hypersurfaces
of $\wi.$

\begin{theorem} \label{th-final}
Let $\Sigma$ be an ideal Einstein hypersurface of $\wi,$ $n>3.$
Then, $\Sigma$ has a constant number of distinct principal curvatures ---
which is two or three --- all with constant multiplicities.
Moreover, $\Sigma$ is trivial
if and only if it has two distinct principal curvatures everywhere.
If so, $\Sigma$ is necessarily rotational.
\end{theorem}

The paper is organized as follows. In Section \ref{sec-prelimminaries},
we set some notation and formulae. In Section \ref{sec-lemmas}, we establish three key lemmas.
In Section \ref{sec-csc}, we consider rotational CSC hypersurfaces of $\wi$ and prove Theorems \ref{th-existenceCSC01} and  \ref{th-CSC}.
Section \ref{sec-einstein} is reserved for the proofs of Theorems \ref{th-betazero} and \ref{th-einstein}. In  Section \ref{sec-graphs},
we introduce the  $(\phi,f_s)$-graphs we mentioned, establishing some
of its fundamental properties. In the final Section \ref{sec-last}, we prove
Theorems \ref{L1-grapheinstein}--\ref{th-final}.

\vtt
\emph{Added in proof.} After a preliminary version of this paper was completed, we became acquainted with
the work of V. Borges and A. da Silva \cite{borges-silva}, which  overlaps  with ours on some parts.

\section{Preliminaries} \label{sec-prelimminaries}

\subsection{Ricci tensor}
Given a Riemannian manifold $(M^n,\langle\,,\,\rangle)$ with curvature
tensor $R$ and tangent bundle $TM,$
recall that its \emph{Ricci tensor} is defined as:
\[
{\rm Ric}\,(X,Y):={\rm trace}\,\{Z\mapsto R(Z,X)Y\}, \,\,\, \,\, (X, Y)\in TM\times TM.
\]

We say that a frame $\{X_1\,,\dots ,X_n\}\subset TM$ \emph{diagonalizes} the Ricci tensor
of  $(M,\langle\,,\,\rangle)$ if
${\rm Ric}\,(X_i,X_j)=0\, \forall i\ne j\in\{1,\dots ,n\}.$
Notice that,  when $M$ is an Einstein manifold, any orthogonal
frame $\{X_1\,,\dots ,X_n\}$ in $TM$ diagonalizes its Ricci tensor.

\subsection{Hypersurfaces}
Let $\Sigma$ be an oriented hypersurface of a Riemannian manifold $\overbar M^{n+1}$ whose
unit normal field we denote by $N.$  Let $A$ be the shape operator of $\Sigma$ with respect to $N,$ that is,
$$
AX=-\overbar\nabla_XN,  \,\, X\in T\Sigma,
$$
where $\overbar\nabla$ is the Levi-Civita connection of $\overbar M^{n+1}.$

For such a $\Sigma$, and for $X,Y, Z, W\in T\Sigma$, one has the \emph{Gauss equation}:
\begin{equation}\label{eq-gauss}
\langle R(X,Y)Z,W\rangle = \langle\overbar R(X,Y)Z,W\rangle+\langle AX,W\rangle\langle AY,Z\rangle-\langle AX,Z\rangle\langle AY,W\rangle,
\end{equation}
where $R$ and $\overbar R$ denote the curvature tensors of $\Sigma$ and $\overbar M^{n+1},$  respectively.

Let $\lambda_1\,, \dots, \lambda_n$ be the principal curvatures of $\Sigma$, and
$\{X_1\,, \dots, X_n\}$ the corresponding orthonormal frame of principal directions, i.e.,
$$
AX_i=\lambda_iX_i\,, \,\,\,\, \langle X_i,X_j\rangle=\delta_{ij}\,.
$$
In this setting, for $1\le k\le n,$ the Gauss equation yields
\begin{equation}\label{eq-gauss01}
\langle R(X_k,Y)Z,X_k\rangle = \langle\overbar R(X_k,Y)Z,X_k\rangle+\lambda_k\langle AY,Z\rangle-\lambda_k^2\langle X_k,Z\rangle\langle X_k,Y\rangle.
\end{equation}

Thus,   for the Ricci tensor ${\rm Ric}$ of $\Sigma$ one has
\begin{equation}\label{eq-ricci01}
{\rm Ric}(Y,Z)=\sum_{k=1}^{n}\langle\overbar R(X_k,Y)Z,X_k\rangle+H\langle AY,Z\rangle-\sum_{k=1}^{n}\lambda_k^2\langle X_k,Z\rangle\langle X_k,Y\rangle,
\end{equation}
where $H={\rm trace}\, A$ is the (non normalized) mean curvature of $\Sigma.$
In particular, if we set $Y=X_i$ and $Z=X_j$\,, we have
\begin{equation}\label{eq-ricci02}
{\rm Ric}(X_i\,,X_j)=\sum_{k=1}^{n}\langle\overbar R(X_k,X_i)X_j,X_k\rangle+H\delta_{ij}\lambda_i-\delta_{ij}\lambda_i^2\,.
\end{equation}

\subsection{Hypersurfaces in $\wi$}
In the above setting, let us consider the particular case when $\overbar M^{n+1}$ is the warped product
$I\times_\omega\Q_\epsilon^n,$ where $I\subset\R$ is an open interval,
$\omega$ is a positive differentiable function defined on $I,$
and $\Q_\epsilon^n$ is one of the simply connected space forms of constant curvature
$\epsilon\in\{0,1,-1\}$: $\R^n$ ($\epsilon =0$),
$\s^n$ ($\epsilon =1$), or $\h^n$ ($\epsilon =-1$).

Recall that the Riemannian metric
of $\wi$, to be denoted by $\langle\,,\, \rangle,$ is:
\[
dt^2+\omega^2ds_\epsilon^2,
\]
where $dt^2$ and $ds_\epsilon^2$ are the standard Riemannian metrics of $\R$ and
$\Q_\epsilon^n,$  respectively.

We call \emph{horizontal} the fields of $T(\wi)$ which are tangent to $I,$ and
\emph{vertical} those which are tangent to $\Q_\epsilon^n.$ The gradient of
the projection $\pi_I$ on the first factor $I$ will be denoted by $\partial_t$\,.

Setting $\overbar\nabla$ and $\widetilde\nabla$\, for the Levi-Civita connections of
$\wi$ and $\Q_\epsilon^n,$ respectively,
for any  \emph{vertical} fields $X, Y\in T\Q_\epsilon^n,$
the following identities hold
(see \cite[Lema 7.3]{bishop-oneill}):
\begin{equation}\label{eq-connectionwarped}
  \begin{aligned}
    \overbar\nabla_XY &= \widetilde\nabla_XY-({\omega'}/{\omega})\langle X,Y\rangle\partial_t\,.\\
    \overbar{\nabla}_X\partial_t &=\overbar{\nabla}_{\partial_t}X=({\omega'}/{\omega}) X.\\
    \overbar{\nabla}_{\partial_t}\partial_t &= 0.
  \end{aligned}
\end{equation}

The curvature tensor $\overbar R$ of $I\times_\omega \,\Q_\epsilon^n$ is given by (cf. \cite{lawn-ortega})
\begin{eqnarray} \label{eq-barcurvaturetensor}
  \langle\overbar R(X,Y)Z,W\rangle  &=& (\alpha\circ\pi_I)(\langle X,Z\rangle\langle Y,W\rangle-\langle X,W\rangle\langle Y,Z\rangle) \nonumber \\
                            && +(\beta\circ\pi_I)(\langle X,Z\rangle\langle Y,\partial_t\rangle\langle W,\partial_t\rangle-\langle Y,Z\rangle\langle X,\partial_t\rangle\langle W,\partial_t\rangle \\
                            && -\langle X,W\rangle\langle Y,\partial_t\rangle\langle Z,\partial_t\rangle+\langle Y,W\rangle\langle X,\partial_t\rangle\langle Z,\partial_t\rangle), \nonumber
\end{eqnarray}
where $X,Y, Z, W\in T(I\times_\omega\,\Q_\epsilon^n),$ and $\alpha$ and $\beta$ are the following functions on $I$:
\begin{equation} \label{eq-def-alpha-beta}
\alpha:=\frac{(\omega')^2-\epsilon}{\omega^2} \quad\text{and}\quad \beta:= \frac{\omega''}{\omega}-\alpha.
\end{equation}

It is easily seen from \eqref{eq-barcurvaturetensor} that $I\times_\omega\q$ has
constant sectional curvature if and only if $\alpha$ is constant and $\beta$ vanishes on $I.$
Also, a direct computation yields:
\begin{equation}\label{eq-alphaprime}
\alpha'=\frac{2\omega'}{\omega}\beta.
\end{equation}

Given a hypersurface $\Sigma$ of $\wi,$
the \emph{height} function $\xi$ and the \emph{angle} function $\theta$ of $\Sigma$
are defined as
$$
\xi:=\pi_{\scriptscriptstyle I}|_\Sigma \quad\text{and}\quad \Theta(x):=\langle N(x),\partial_t\rangle, \,\, x\in\Sigma.
$$

The gradient of $\xi$ on $\Sigma$ is denoted by $T,$ so that
\begin{equation} \label{eq-Tandtheta}
T=\partial_t-\Theta N.
\end{equation}

\begin{definition}
A connected Einstein  hypersurface $\Sigma$ of $\wi$ on which
$(\beta\circ\xi)T$ never vanishes will be called \emph{ideal}.
\end{definition}

Regarding the gradient of $\theta,$ let us notice first that,
from the equalities \eqref{eq-connectionwarped}, for all $X\in T(\wi),$ one  has
\begin{equation} \label{eq-connectionwarp2}
\overbar\nabla_X\partial_t=\overbar\nabla_{X-\langle X,\partial_t\rangle\partial_t}\partial_t=
\frac{\omega'}{\omega}\left(X-\langle X,\partial_t\rangle\partial_t\right),
\end{equation}
where, by abuse of notation, we are writing $\omega\circ\xi=\omega$ and
$\omega'\circ\xi=\omega'$.
Therefore, for any  $X\in T\Sigma,$ we have
\[
X(\theta)=\langle\overbar\nabla_XN,\partial_t\rangle+\langle N,\overbar\nabla_X\partial_t\rangle=-\langle AX,T\rangle-\frac{\omega'}{\omega}\theta\langle T,X\rangle=
-\langle AT,X\rangle-\frac{\omega'}{\omega}\theta\langle T,X\rangle.
\]
Hence, the gradient of $\theta$ is
\begin{equation} \label{eq-gradthetawarp}
\nabla\theta=-\left(A+\frac{\omega'}{\omega}\theta\,{\rm Id}\right)T,
\end{equation}
where ${\rm Id}$ stands for the identity map of $T\Sigma.$

The following concept will play a fundamental role in this paper.

\begin{definition}
We say that a hypersurface  $\Sigma$ of $\wi$ has the $T$-\emph{property} if $T$ is a principal direction
of $\Sigma,$ i.e., $T$ never vanishes and there is a differentiable real function  $\lambda$ on $\Sigma$ such that
$AT=\lambda T$, where $A$ is the shape operator of $\Sigma.$
\end{definition}

Finally, it follows from \eqref{eq-barcurvaturetensor} that, for any hypersurface
$\Sigma$ of $\wi,$ the
Gauss equation \eqref{eq-gauss} takes the form:
\begin{eqnarray} \label{eq-gauss02}
  \langle R(X,Y)Z,W\rangle  &=& (\alpha\circ\xi)(\langle X,Z\rangle\langle Y,W\rangle-\langle X,W\rangle\langle Y,Z\rangle) \nonumber \\
                            &&+ (\beta\circ\xi)(\langle X,Z\rangle\langle Y,T \rangle\langle W,T \rangle-\langle Y,Z\rangle\langle X,T \rangle\langle W,T \rangle \nonumber \\
                            &&- \langle X,W\rangle\langle Y,T \rangle\langle Z,T\rangle+\langle Y,W\rangle\langle X,T\rangle\langle Z,T\rangle)  \\
                            &&+ \langle AX,W\rangle\langle AY,Z\rangle-\langle AX,Z\rangle\langle AY,W\rangle. \nonumber
\end{eqnarray}


\section{Three Key Lemmas} \label{sec-lemmas}

Given a hypersurface $\Sigma\subset I\times_\omega\Q_\epsilon^n,$ any nonempty transversal intersection
$$
\Sigma_t:=\Sigma\transv(\{t\}\times\Q_\epsilon^n)
$$
is a hypersurface of $(\{t\}\times\Q_\epsilon^n,ds_\epsilon^2)$
which we call the \emph{vertical section} of $\Sigma$ at height $t.$ If $\Sigma$ is oriented with
unit normal $N,$ it is easily seen that
$N-\theta\partial_t$ is tangent to
$\{t\}\times\Q_\epsilon^n$ and orthogonal to $\Sigma_t$\,, so that
$\Sigma_t$ is orientable as well.
Denoting also by $\langle\,,\,\rangle_{\epsilon}$ the Riemannian metric $ds_\epsilon^2$ of $\Q_\epsilon^n,$ one has
\[
\omega^2\langle N-\theta\partial_t,N-\theta\partial_t\rangle_{\epsilon}=\langle N-\theta\partial_t,N-\theta\partial_t\rangle=1-\theta^2=\|T\|^2.
\]
Hence, a unit normal
to $\Sigma_t$ in $(\{t\}\times\Q_\epsilon^n,ds_\epsilon^2)$ is given by
\begin{equation} \label{eq-Nt}
N_t=-\frac{\omega(t)(N-\theta\partial_t)}{\|T\|}\,\cdot
\end{equation}

We shall denote by $A_t$ the
shape operator of $\Sigma_t$ with respect to $N_t\,,$
that is,
$$
A_tX=-\widetilde\nabla_XN_t\,, \,\,\, X\in T\Sigma_t\subset T\Sigma,
$$
where $\widetilde\nabla$ stands for the Levi-Civita connection of
$(\Q_\epsilon^n,ds_\epsilon^2).$

Our first lemma establishes the relation between the shape operator
$A$ of a hypersurface $\Sigma\subset\wi$ with the shape operator $A_t$ of a vertical section
$\Sigma_t\subset(\{t\}\times\Q_\epsilon^n,ds_\epsilon^2).$

\begin{lemma} \label{lem-verticalsection}
For $n\ge 2,$ let $\Sigma$ be a hypersurface of $I\times_\omega\Q_\epsilon^n,$  and let
$\Sigma_t$ be one of its vertical sections (considered as a hypersurface of $(\{t\}\times\Q_\epsilon^n,ds_\epsilon^2)$).
Then, the following identity holds:
\begin{equation}  \label{eq-At001}
\langle A_t X,Y\rangle=-\frac{\omega}{\|T\|}\left(\langle AX,Y\rangle+\frac{\theta\omega'}{\omega}\langle X,Y\rangle\right)
\,\,\,\, \forall X, Y\in T\Sigma_t\subset T\Sigma.
\end{equation}
Consequently, if \,$T$ is a principal direction of $\Sigma$ along $\Sigma_t$
with corresponding principal curvature $\lambda_1,$ the principal curvatures $\lambda_2,\dots, \lambda_n$
of $\Sigma$ are given by
\begin{equation} \label{eq-principalcurvatureslemma}
\lambda_i=-\frac{\|T\|\lambda_i^t+\omega'\theta}{\omega}\,,  \,\,\,\,\, i=2,\dots, n,
\end{equation}
where  $\lambda_i^t$ denotes the $i$-th principal curvature of $\Sigma_t.$
\end{lemma}
\begin{proof}
Given $X\in T\Sigma_t,$ it follows from the first identity \eqref{eq-connectionwarped} that
$\widetilde\nabla_XN_t=\overbar\nabla_XN_t,$ for $\langle X,N_t\rangle=0.$
However,
from \eqref{eq-Nt}, we have
\[
-\overbar\nabla_XN_t=X({\omega}/{\|T\|})(N-\theta\partial_t)+
({\omega}/{\|T\|})(\overbar\nabla_XN-X(\theta)\partial_t-\theta\overbar\nabla_X\partial_t).
\]
Thus, since $\overbar\nabla_X\partial_t=(\omega'/\omega)X$ and $AX=-\overbar\nabla_XN,$ for all $Y\in T\Sigma_t$\,, one has
\[
\langle A_tX,Y\rangle=-\langle \widetilde\nabla_XN_t,Y\rangle=
-\langle\overbar\nabla_XN_t,Y\rangle=-\frac{\omega}{\|T\|}\langle AX+\theta(\omega/\omega')X,Y\rangle,
\]
as we wished to prove.

Now, assuming that $T$ is a principal direction of $\Sigma$
along $\Sigma_t,$  we have that $T\Sigma_t$ is invariant by $A.$
This, together with \eqref{eq-At001}, then yields
\[
A_t=-\frac{\omega}{\|T\|}\left(A+\theta({\omega'}/{\omega}){\rm Id}\right){|_{T\Sigma_t}}\,,
\]
where $\rm{Id}$ denotes the identity map of $T\Sigma_t.$

From this last equality,
one easily concludes that any principal direction $X_i$ of $\Sigma_t$ with
corresponding principal curvature $\lambda_i^t$ is a principal direction of
$\Sigma$ with corresponding principal curvature $\lambda_i$ as in
\eqref{eq-principalcurvatureslemma}.
\end{proof}

Now, we establish a necessary and sufficient condition
for an arbitrary  hypersurface of $\wi$ to have the $T$-property.

\begin{lemma} \label{lem-T}
For $n>2,$ let $\Sigma$  be a hypersurface of $I\times_\omega \,\Q_\epsilon^n$ on which
$(\beta\circ\xi)T$  never vanishes.  Then,  $\Sigma$ has the
$T$-property if and only if the principal directions $X_1\,, \dots ,X_n$ of $\Sigma$  diagonalize its Ricci tensor.
In particular, if $\Sigma$ is Einstein, it has the $T$-property.
\end{lemma}

\begin{proof}
Choosing $i, j, k\in\{1,\dots ,n\}$ with $i\ne j\ne k\ne i,$ we have from \eqref{eq-barcurvaturetensor} that
\[
\langle\overbar R(X_k,X_i)X_j,X_k\rangle= -\beta\langle X_i,T\rangle\langle X_j,T\rangle.
\]
Combining this equality with \eqref{eq-ricci02}, we get
$${\rm Ric}(X_i\,,X_j)=\sum_{k=1}^{n}\langle\overbar R(X_k,X_i)X_j,X_k\rangle=(2-n)(\beta\circ\xi)\langle X_i\,, T\rangle\langle X_j\,, T\rangle,$$
from which the result follows.
\end{proof}

Finally, as a direct consequence of equation \eqref{eq-ricci02}, we obtain  necessary and sufficient
conditions on the principal curvatures of a hypersurface
$\Sigma\subset\wi$ with the $T$-property to be an Einstein manifold.

\begin{lemma} \label{lem-T-einstein}
For $n>2,$ let $\Sigma$  be a hypersurface of \,$I\times_\omega \,\Q_\epsilon^n$
having the $T$-property. Let $\lambda_1\,, \dots, \lambda_n$ be the principal curvatures of $\Sigma$, and
let $\{X_1\,, \dots, X_n\}$ be the corresponding orthonormal frame of principal directions
with $X_1 = T / ||T||$. Then, $\Sigma$ is a $\Lambda$-Einstein hypersurface if and only if
the following equalities hold:
\begin{equation}\label{T-einstein-lambda1}
  \begin{aligned}
    \lambda_1^2 - H\lambda_1 + (n-1)((\beta \circ \xi) ||T||^2 + (\alpha \circ \xi)) + \Lambda &= 0,\\[1ex]
    \lambda_i^2 - H \lambda_i + (\beta \circ \xi) ||T||^2 + (n-1) (\alpha \circ \xi) + \Lambda &=0, \,\,\,\, 2\le i\le n.
  \end{aligned}
\end{equation}
\end{lemma}

\begin{proof}
The proof is straightforward considering the identity \eqref{eq-ricci02}
for $X_i=X_j=X_1$ and $X_i=X_j \neq X_1,$ respectively.
\end{proof}

\section{Rotational CSC Hypersurfaces of $I\times_\omega\Q_\epsilon^n$} \label{sec-csc}

In \cite{dillenetal}, the authors introduced a class of  hypersurfaces, called
\emph{rotational}, which was considered in \cite{manfio-tojeiro} to construct and classify CSC hypersurfaces
of $\R\times\q$ (interchanging factors).
Such a  rotational hypersurface $\Sigma$ is characterized by the fact that
the horizontal projections on $\q$ of its  vertical sections
 constitute a parallel family $\mathscr F$ of totally umbilical hypersurfaces. In this way,
 $\Sigma$ is called \emph{spherical} if the members of $\mathscr F$
 are geodesic spheres of $\q,$ \emph{parabolic} if they are horospheres
of $\h^n$, and \emph{hyperbolic} if they are equidistant hypersurfaces of $\h^n.$

\begin{table}[hbtp]%
\centering %
\begin{tabular}{cccc}
\toprule %
$f(x)$ &  $\epsilon$   & rotational type \\\otoprule %
$\cos x$ & $\phantom +1$ & spherical \\
$\sinh x$ &  $ -1$& spherical \\
$x$ &  $\phantom{-}0$ & spherical \\
$x$ & $-1$& parabolic \\
$\cosh x$ & $-1$& hyperbolic \\\bottomrule
\end{tabular}
\vtt
\caption{Definition of $f.$}
\label{table-f}
\end{table}

Given an interval $I\subset\R,$
any rotational hypersurface $\Sigma\subset I\times\q$ can be parametrized by means of a plane curve defined
by two differentiable functions on an open interval.
Here, we  denote these functions by $\phi=\phi(s)$ and $\xi=\xi(s)$
(see \cite{manfio-tojeiro} for details).
In this setting,  we can replace $I\times\q$ by $\wi$ and consider $\Sigma$ as a hypersurface of
the latter. Then, assuming that either
\[
(\xi')^2+(\omega(\xi)\phi')^2=1 \quad\text{or}\quad (\xi')^2+\left(\omega(\xi)\frac{\phi'}{\phi}\right)^2=1\,\,\,(\text{if $\Sigma$ is parabolic}),
\]
and performing computations analogous to the ones in \cite[Section 4]{manfio-tojeiro}, one concludes that the
metric $d\sigma^2$ on $\Sigma$ induced by the warped metric of $\wi$ is  a warped metric as well,
which is given by
\begin{equation} \label{eq-rotationalmetric}
d\sigma^2=ds^2+(\omega(\xi)f(\phi))^2du^2,
\end{equation}
where $f,$ as defined in Table \ref{table-f},
depends on $\epsilon$ and the rotational type of $\Sigma$,
and $du^2$ is the standard metric of
$\s^{n-1},$  $\R^{n-1}$ or $\h^{n-1},$  according to
whether $\Sigma$ is spherical, parabolic or hyperbolic,
respectively.

Let $\Sigma$ be a rotational hypersurface of $\wi$ whose metric
$d\sigma^2$ is given by \eqref{eq-rotationalmetric}.
By \cite[Proposition 4.6]{manfio-tojeiro}, $\Sigma$ has constant sectional curvature
$c\in\R$ if its warping function $(\omega\circ\xi)f(\phi)$ coincides with the
function $\psi=\psi(s),$ defined in Table \ref{table-psi}, which depends on
$\epsilon,$ the sign of $c,$ and the rotational type of $\Sigma.$
\begin{table}[thb]%
\centering %
\begin{tabular}{cccc}
\toprule %
$\psi(s)$ & sign of $c$ & $\epsilon$   & rotational type \\\otoprule %
$\frac{1}{\sqrt c}\sin\sqrt c s$ & $c>0$ & $0,\pm 1$& spherical \\
$ s$ & $c=0$ & $0,\pm 1$& spherical \\
$\frac{1}{\sqrt{-c}}\sinh\sqrt{-c}s$ & $c<0$ & $0,-1$& spherical \\\midrule
$e^{\sqrt{-c}s}$ & $c<0$ & $-1$& parabolic \\
cte & $c=0$ & $-1$& parabolic \\\midrule
$\frac{1}{\sqrt{-c}}\cosh\sqrt{-c}s$ & $c<0$ & $-1$& hyperbolic \\\bottomrule
\end{tabular}
\vtt
\caption{Definition of $\psi.$}
\label{table-psi}
\end{table}

Therefore, given $c\in\R,$ our problem of finding a rotational hypersurface of $\wi$
with constant sectional curvature $c$ reduces to solving either the systems
\begin{equation} \label{eq-system}
\left\{
\begin{array}{rcc}
(\xi')^2+(\omega(\xi)\phi')^2 &=&1\\[1ex]
\omega(\xi)f(\phi) &=&\psi
\end{array}
\right.
\quad\text{or}\,\,\quad
\left\{
\begin{array}{rcl}
(\xi')^2+\left(\omega(\xi)\frac{\phi'}{\phi}\right)^2&=&1\\[1ex]
\omega(\xi)f(\phi) &=&\psi,
\end{array}
\right.
\end{equation}
where the second one is considered only in the parabolic case.

Let us assume that $\omega:I\rightarrow\omega(I)$ is a diffeomorphism, and
that $f(\phi)$ does not vanish. Under these assumptions,  the second equation of these
systems gives that
\begin{equation} \label{eq-xi}
\xi=\chi\left(\frac{\psi}{f(\phi)}\right),
\end{equation}
where $\chi:\omega(I)\rightarrow I$ is the inverse of $\omega.$ Of course,  equality \eqref{eq-xi}
makes sense only if  the range of ${\psi}/{f(\phi)}$ is contained in $\omega(I).$ Assuming that, taking
the derivative, and substituting $\xi'$ in the first equation of the system, we have that
$\phi$ must be a solution of  a nonlinear first order ODE
in the following form:
\begin{equation}\label{eq-EDOphi}
a_2(s,y)(y')^2+a_1(s,y)y'+a_0(s,y)=0,
\end{equation}
where each coefficient $a_i$ is a differentiable function.

Let us assume that, for some
pair $(s_0,y_0),$ we have
\begin{equation}\label{eq-Delta}
\Delta(s_0,y_0):=a_1^2(s_0,y_0)-4a_2(s_0,y_0)a_0(s_0,y_0)>0.
\end{equation}
In this case,  \eqref{eq-EDOphi} can be written as
\begin{equation}\label{eq-EDOphi2}
y'=F(s,y),
\end{equation}
where $F$ is a differentiable function in a neighborhood
of $(s_0,y_0).$

Therefore, if the inequality \eqref{eq-Delta} holds, we can apply
Picard's Theorem  and conclude that
there exists a solution
$y=\phi(s)$ to \eqref{eq-EDOphi2}, and so to \eqref{eq-EDOphi},
satisfying the initial condition $\phi(s_0)=y_0$ (see \cite[pg 82]{ince} for details).
Applying it to \eqref{eq-xi}, we obtain the function $\xi$,
so that $(\phi,\xi)$ is a solution of \eqref{eq-system}.

The following result, which will be considered in the proof
of Theorem \ref{th-existenceCSC01}, summarizes the above discussion.

\begin{proposition} \label{prop-solutionsystem}
Let $\omega:I\rightarrow\omega(I)$ be a diffeomorphism. Given
$c\in\R,$ suppose that there exists $s_0\in\R$ such that the function
$\xi$ given in \eqref{eq-xi007} is well defined in an open interval
$I_0\owns s_0$\,. If, for some $y_0\in\R,$ one has $\Delta(s_0,y_0)>0,$ then
the system \eqref{eq-system} has a solution $(\phi,\xi)$ defined in an open
interval contained in $I_0$\,. Consequently, there is a hypersurface
of constant sectional curvature $c$ in the
corresponding warped product $I\times_\omega\q.$
\end{proposition}

\begin{proof}[Proof of Theorem \ref{th-existenceCSC01}]

We can assume,
without loss of generality,  that
$$\omega(I)=(0,\delta), \,\,\,  0<\delta\le+\infty,$$
for $\omega>0.$
Then, the hypothesis on $\omega'$ yields
\begin{equation}\label{eq-xiderivative}
|\chi'(u)|<1 \,\,\, \forall u\in (0,\delta).
\end{equation}

We will consider separately the cases in the statement (according to
$\epsilon,$ $c,$ and the rotational type of the hypersurface) to show that,
in any of them, there exists a pair $(s_0,y_0)$ such that $\xi(s)$ (as given in \eqref{eq-xi})
is well defined in  an open interval $I_0\owns s_0$\,, and  $\Delta(s_0,y_0)>0.$
The result, then, will follow
from Proposition \ref{prop-solutionsystem}.

\vtt
{\bf Case 1:} \underline{$\epsilon=-1$, $c\in(-\infty,+\infty)$, spherical-type}
\vtt

In this setting,
$f(x)=\sinh x$ and $\psi(s)$ is as in the  first three lines
of Table \ref{table-psi}, according to the sign of $c.$ Then, we have
\begin{equation} \label{eq-xi007}
\xi(s)=
\left\{
\begin{array}{lcc}
\chi\left(\frac{\sin(\sqrt cs)}{\sqrt c\sinh(\phi(s))}\right) &\text{if}& c>0.\\[2ex]
\chi\left(\frac{s}{\sinh(\phi(s))}\right) &\text{if}& c=0.\\[2ex]
\chi\left(\frac{\sinh(\sqrt {-c}s)}{\sqrt {-c}\sinh(\phi(s))}\right) &\text{if}& c<0.
\end{array}
\right.
\end{equation}

In any case,  we  choose  $s_0$  positive and sufficiently small, and
$y_0=\phi(s_0)$ satisfying  $\sinh y_0>1$,  so that $\xi$ is well defined in an  open
interval $I_0\owns s_0$\,, $I_0\subset (0,+\infty).$

A straightforward calculation gives that, at $(s_0,y_0),$
the coefficients $a_i$ (defined in \eqref{eq-EDOphi})
satisfy $a_2(s_0,y_0)>0$ and
\[
a_0(s_0,y_0)=
\left\{
\begin{array}{lcc}
(\chi'(u_0))^2\frac{\cos^2(\sqrt cs_0)}{\sinh^2(y_0)}-1 &\text{if}& c>0.\\[2ex]
\frac{(\chi'(u_0))^2}{\sinh^2(y_0)}-1 &\text{if}& c=0.\\[2ex]
(\chi'(u_0))^2\frac{\cosh^2(\sqrt{-c}s_0)}{\sinh^2(y_0)}-1  &\text{if}& c<0,
\end{array}
\right.
\]
where $u_0$ is the argument of $\chi$ in \eqref{eq-xi007} for
$s=s_0$\,.
From the choice of $y_0$ and \eqref{eq-xiderivative}, we have  $a_0(s_0,y_0)<0$
in all three cases, which implies that  $\Delta(s_0,y_0)>0.$

\vtt
{\bf Case 2:} \underline{$\epsilon=0$, $c\in(-\infty,+\infty)$, spherical-type}
\vtt

The reasoning for this case is completely analogous to the one for Case 1.
We have just to replace, in that argument, the function $f(x)=\sinh x$ by the identity function $f(x)=x.$

\vtt
{\bf Case 3:} \underline{$\epsilon=1$, $c>0$, spherical-type}
\vtt

Now, the functions
$f$ and $\psi$ are
$f(x)=\cos x$ and
$\psi(s)=\sin(\sqrt cs)/\sqrt c,$
in which case equality  \eqref{eq-xi} takes the form
\[
\xi(s)=\chi\left(\frac{\sin(\sqrt cs)}{\sqrt c\cos\phi(s)}\right)\cdot
\]
Hence, choosing  $s_0$ and $y_0=\phi(s_0)$  positive and
sufficiently small, one has
\[
u_0:=\frac{\sin(\sqrt cs_0)}{\sqrt c\cos y_0}\in\omega(I),
\]
which implies that  $\xi(s)$ is well defined in  an
open interval $I_0\owns s_0$\,, $I_0\subset (0,+\infty).$
Also, a direct computation yields
$a_2(s_0,y_0)>0$ and
 \[a_0(s_0,y_0)=(\chi'(u_0))^2\frac{\cos^2(\sqrt cs_0)}{\cos^2(y_0)}-1.\]

Thus, assuming $y_0<\sqrt cs_0$ and considering \eqref{eq-xiderivative}, we conclude that
$a_0(s_0,y_0)<0.$ In particular, $\Delta(s_0,y_0)>0.$

\vtt
{\bf Case 4:} \underline{$\epsilon=-1$, $c\le 0$, parabolic-type}
\vtt

For this case, $f(x)=x$ and $\psi(s)$ is as in the fourth and fifth
lines of Table \ref{table-psi}, according to the sign of $c.$ Hence,
\begin{equation} \label{eq-xi008}
\xi(s)=
\left\{
\begin{array}{lcc}
\chi\left(\frac{e^{\sqrt{-c}s}}{\phi(s)}\right) &\text{if}& c<0.\\[2ex]
\chi\left(\frac{a}{\phi(s)}\right) &\text{if}& c=0,
\end{array}
\right.
\end{equation}
where $a$ is a positive constant. Setting $s_0=0,$ in any of the above cases, we
can choose a sufficiently large $y_0=\phi(0)$ such that $\xi$ is well defined
in a neighborhood of $0.$ For $c<0,$ we shall also assume $y_0>-c.$

Writing $u_0$ for the argument of $\chi$ in \eqref{eq-xi008},
we have  $a_2(s_0,y_0)>0$ and
\[
a_0(s_0,y_0)=
\left\{
\begin{array}{lcc}
\frac{-c(\chi'(u_0))^2}{y_0^4}-1 &\text{if}& c<0.\\[2ex]
-1 &\text{if}& c=0.
\end{array}
\right.
\]

So, considering \eqref{eq-xiderivative} (only for  $c<0$), we have
that $\Delta(s_0,y_0)>0$ in both cases.

\vtt

{\bf Case 5:} \underline{$\epsilon=-1$, $c<0$, hyperbolic-type}

\vtt
The functions $f$ and $\psi$ are $f(x)=\cosh x$ and $\psi(s)=\cosh s.$ Thus,
we have
\begin{equation} \label{eq-xifinal}
\xi(s)=\chi\left(\frac{\cosh(\sqrt{-c}s)}{\sqrt{-c}\cosh\phi(s)}\right).
\end{equation}

As in the previous case, by  choosing $s_0=0$ and $y_0=\phi(s_0)$ sufficiently large, we have
that $\xi$ is well defined in a neighborhood of $0.$ Again,
setting $u_0$ for the argument of $\xi$ in \eqref{eq-xifinal},
we have $a_2(s_0,y_0)>0$ and
\[
a_0(s_0,y_0)=\frac{(\chi'(u_0)\sinh(\sqrt{-c}s_0))^2}{\cosh^4(y_0)}-1=-1<0,
\]
which implies that $\Delta(s_0,y_0)>0.$
\end{proof}

Besides having totally umbilical vertical sections,
a notable fact of rotational hypersurfaces in $I\times\q$
is that they have the $T$-property (see \cite{dillenetal}).
Let us see that the same is true for hypersurfaces of
warped products $\wi.$
To that end, we first consider the diffeomorphism $G:I\rightarrow J=G(I)$ such that
$G'=1/\omega.$ As can be easily checked, the map
\begin{equation} \label{eq-varphi}
\begin{array}{cccc}
\varphi\colon & \wi   & \rightarrow & J\times\Q_\epsilon^n\\[1ex]
              & (t,p) & \mapsto     & (G(t),p)
\end{array}
\end{equation}
is a conformal diffeomorphism.  This, together with \cite[Lemma 3]{manfio-tojeiro-veken}, gives that a hypersurface
$\Sigma\subset\wi$ has the $T$-property if and only if $\varphi(\Sigma)$
has the $T$-property as a hypersurface of $J\times\Q_\epsilon^n$.
Also, since  $\varphi$ fixes the second factor $\Q_\epsilon^n$ pointwise,
a vertical section $\Sigma_t\subset\Sigma$ is totally umbilical
in $(\{t\}\times\Q_\epsilon^n,ds_\epsilon^2)$ if and only if
$\varphi(\Sigma_t)\subset\varphi(\Sigma)$ is totally umbilical
in $((\{G(t)\}\times\Q_\epsilon^n,ds_\epsilon^2).$ Consequently,
$\Sigma$ is rotational if and only if $\varphi(\Sigma)$ is rotational.

The following result, which will be applied in the
proof of Theorem \ref{th-CSC},  follows directly from the above considerations,
the identity \eqref{eq-principalcurvatureslemma}, and
\cite[Theorem 2]{dillenetal} (see also \cite[Remark 7.4]{manfio-tojeiro}).

\begin{proposition} \label{prop-TandRotational}
For $n\ge 3,$ let $\Sigma\in\wi$  be a hypersurface having the $T$-property.
Then, $\Sigma$ is rotational if and only if
any vertical section $\Sigma_t$ of $\Sigma$ is  totally
umbilical in $(\{t\}\times\Q_\epsilon^n,ds_\epsilon^2).$
\end{proposition}


\begin{proof}[Proof of Theorem \ref{th-CSC}]
As $\Sigma$ has constant sectional curvature, it is an Einstein manifold. Thus, by Lemma \ref{lem-T},
$T$ is a principal direction of $\Sigma,$ so that we can set
\[
\{X_1\,, \dots, X_n\}\subset T\Sigma
\]
for the orthonormal frame of principal directions of $\Sigma$ with $X_1=T/\|T\|.$

Writing  $c$ for the sectional curvature of $\Sigma,$
$\lambda_1$ for its principal curvature corresponding to $T,$  and
$\lambda_2,\dots,\lambda_n$ for the principal curvatures corresponding to the
vertical principal directions of $\Sigma,$
we get from Gauss equation \eqref{eq-gauss02} that, for all $2\le i,j\le n,$
$i\ne j,$ the following identities hold:
\begin{equation}\label{eq-gaussproof}
  \begin{aligned}
    (\alpha\circ\xi) +c &= \lambda_i\lambda_j\\
    (\alpha\circ\xi)+c  &=\lambda_1\lambda_i-(\beta\circ\xi)\|T\|^2.
  \end{aligned}
\end{equation}

Suppose that, for some $x\in\Sigma,$ $\alpha(\xi(x))+c=0.$ Then, from the first identity in \eqref{eq-gaussproof},
$\lambda_i(x)$ is nonzero for at most one index $i\in\{2,\dots ,n\}.$ Thus, choosing $i\in\{2,\dots ,n\}$ such that $\lambda_i(x)=0$,
we have from the second identity in \eqref{eq-gaussproof} that  $(\beta\circ\xi)\|T\|$  vanishes at $x,$
which is contrary to our hypothesis. Hence,
$(\alpha\circ\xi) +c$ never vanishes on $\Sigma.$

Since we are assuming  $n>3,$ the first identity in \eqref{eq-gaussproof}  implies that
the last $n-1$ principal curvatures of $\Sigma$ are all equal and nonzero.
This, together with
\eqref{eq-principalcurvatureslemma}, gives that any vertical section $\Sigma_t$ of $\Sigma$ is totally umbilical
in $(\{t\}\times\Q_\epsilon^n,ds_\epsilon^2).$
The result, then, follows from Proposition \ref{prop-TandRotational}.
\end{proof}

\section{Proofs of Theorems \ref{th-betazero} and \ref{th-einstein}} \label{sec-einstein}

\begin{proof}[Proof of Theorem \ref{th-betazero}]
The proof follows closely the one given for \cite[Theorem 3.1]{ryan}.

Considering equality \eqref{eq-alphaprime} and the fact that $T=\nabla\xi,$ it follows
from the hypothesis on $(\beta\circ\xi)T$ that
the derivative of $\alpha\circ\xi$ vanishes on $\Sigma.$ Indeed, given $x\in\Sigma,$
one has
$$
(\alpha\circ\xi)_*(x)=\alpha'(\xi(x))\xi_*(x)=\frac{2\omega'(\xi(x))}{\omega(\xi(x))}\beta(\xi(x))\xi_*(x)=0.
$$
In particular,  $\alpha\circ\xi$ (to be denoted by $\alpha$) is  constant, since  $\Sigma$ is connected.

Let $\{X_1\,,\dots, X_n\}\subset T\Sigma$ be a local  orthonormal frame of principal directions of
$\Sigma.$  It follows from $\eqref{eq-barcurvaturetensor}$ that
$\langle\overbar R(X_i\,, X_k)X_k\,, X_i\rangle=-\alpha\,\, \forall i\ne k\in\{1,\dots ,n\}.$ This  equality, together
with \eqref{eq-ricci02}, gives that any principal direction $\lambda_i$ of
$\Sigma$ is a root of the following quadratic equation:
\begin{equation}\label{eq-quadratic}
s^2-Hs+\sigma=0.
\end{equation}
In particular, there are only two possible values for
each $\lambda_i.$ Hence, we can assume that, for some $k\in \{1,\dots ,n\},$ one has
$$\lambda_1=\cdots =\lambda_k=\lambda \quad\text{and}\quad \lambda_{k+1}=\cdots =\lambda_{n}=\mu,$$
with $\lambda$ and $\mu$ satisfying the identities:
\begin{equation} \label{eq-lm}
\lambda+\mu=H \quad \text{and} \quad \lambda\mu=\sigma.
\end{equation}
Since $H={\rm trace}\, A=k\lambda+(n-k)\mu,$ it follows that
\begin{equation} \label{eq-k}
(k-1)\lambda+(n-k-1)\mu=0.
\end{equation}

Assume  $\sigma>0.$ In this case,  if $\lambda$ and $\mu$ are distinct at some point of $\Sigma$,
they must have the same sign. Thus,
by equality \eqref{eq-k}, we must have $k-1=n-k-1=0,$ which gives $n=2.$ However, we are assuming $n>2.$
Hence, $\Sigma$ is totally umbilical if
$\sigma>0.$ In addition, $\lambda^2=\sigma,$ i.e., $\lambda$ is constant on $\Sigma.$
But, by Gauss equation \eqref{eq-gauss02},
for any pair of orthonormal principal directions $X_i\,, X_j$\,, one has
$K(X_i,X_j)=\lambda^2-\alpha,$
which implies that $\Sigma$ has constant sectional curvature $c=\lambda^2-\alpha.$

Suppose now that $\sigma=0.$ If $\lambda\ne 0$ and $\mu=0,$ we have $H=k\lambda.$ However,
$\lambda$ is a solution of  \eqref{eq-quadratic}. Hence,
$\lambda^2-k\lambda^2=0,$ that is, $k=1,$ which implies
that $\Sigma$ has at most one nonzero principal curvature. This,
together with Gauss equation, gives that $\Sigma$ has
constant sectional curvature $c=-\alpha.$

Finally, assume that $\sigma=\lambda\mu<0,$ in which case $\lambda$ and $\mu$ are distinct and nonzero.
Then, considering \eqref{eq-k}, we conclude that $1<k<n-1,$ that is,
both $\lambda$ and $\mu$ have  multiplicity at least $2.$
In addition, since $H$ is differentiable on $\Sigma$, as distinct roots of
\eqref{eq-quadratic},  $\lambda$ and $\mu$ are also differentiable. This, together
with \eqref{eq-k}, gives that $k$ is differentiable, and so is constant on $\Sigma,$
that is, $\lambda$ and $\mu$ have constant multiplicities.

It  follows from the above considerations and the identities \eqref{eq-lm} and \eqref{eq-k}
that $\lambda$ and $\mu$ are  constant functions on $\Sigma.$
Now, choosing orthonormal pairs $X_1, \, X_2$  and $Y_1\,, Y_2$  satisfying
$AX_i=\lambda X_i$  and $AY_i=\mu Y_i$\,, we get from  Gauss equation that
$$K(X_1,X_2)=\lambda^2-\alpha\ne\mu^2-\alpha=K(Y_1\,,Y_2),$$
which implies that $\Sigma$ has nonconstant sectional curvature.
\end{proof}

\begin{remark} \label{rem-betaT=0}
It is a well known fact that, for any $\delta\in\{-1,0,1\},$
$\Q_\delta^{n+1}$ contains open dense subsets which can
be represented as  warped products $\wi$ (see, e.g., \cite[Section 3.2]{manfio-tojeiro-veken}).
In all these representations, the corresponding function $\alpha$ is constant, so that Theorem \ref{th-betazero}
applies. More precisely, the cases (i) and (ii) occur for any $\delta\in\{-1,0,1\},$   whereas case (iii) occurs
only for $\delta=1.$ Namely, for the set
$\s^{n+1}-\{e_{n+2}\,,-e_{n+2}\}=(0,\pi)\times_{\sin t}\s^n$ (cf. \cite[Theorem 3.1]{ryan}).
\end{remark}

Let us recall that a Riemannian manifold $\overbar M^{n+1}$ is called \emph{conformally flat} if it is
locally conformal to  Euclidean space. Riemannian manifolds
of constant sectional curvature are known to be conformally flat. Conversely,
any conformally flat Einstein manifold has constant sectional curvature.
Another classical result on this subject states that, for $n\ge 4,$  any hypersurface of a
conformally flat manifold $\overbar M^{n+1}$ is also conformally flat if and only if
it has a principal curvature of multiplicity $n-1$ (cf. \cite[chap. 16]{dajczer-tojeiro}).
These facts will be considered in the proof below.

\begin{proof}[Proof of Theorem \ref{th-einstein}]
By Lemma \ref{lem-T},  $\Sigma$ has the $T$-property. Hence,
by Proposition \ref{prop-TandRotational}, the assertions (ii) and (iii) are equivalent. In addition, by
Theorem \ref{th-CSC}, (i) implies (iii). So, it remains to prove that (ii) implies (i).
To that end, let us observe that, from the second equality in \eqref{T-einstein-lambda1},
any of the principal curvatures $\lambda_i$\,, $i\in\{2,\dots ,n\},$
of $\Sigma$ is a root of the
quadratic equation $s^2-Hs+\sigma=0,$ where
\begin{equation} \label{eq-sigma}
\sigma=\Lambda+(n-1)(\alpha\circ\xi)+\|T\|^2(\beta\circ\xi).
\end{equation}

Therefore, at each point of $\Sigma,$ there are at most two
different values for such principal curvatures, that is,
there exists $k\in\{2,\dots ,n\}$ such that
$$\lambda_2=\cdots =\lambda_k=\lambda \quad\text{and}\quad \lambda_{k+1}=\cdots =\lambda_{n}=\mu.$$

Consider a vertical
section $\Sigma_t\subset\Sigma,$ $t\in I.$
It follows from Lemma \ref{lem-verticalsection} that
$X_2\,, \dots , X_n$ are  principal directions of $\Sigma_t$ whose principal curvatures are
\begin{equation} \label{eq-lambdatmut}
\lambda_t:=-\frac{\omega(t)}{\|T\|}\left(\lambda+\frac{\theta\omega'(t)}{\omega(t)}\right) \quad\text{and}\quad
\mu_t:=-\frac{\omega(t)}{\|T\|}\left(\mu+\frac{\theta\omega'(t)}{\omega(t)}\right).
\end{equation}
In particular, $\lambda_t=\mu_t$ if and only if $\lambda=\mu.$

Thus, if all vertical sections $\Sigma_t\subset\Sigma$ are totally umbilical,
then $\Sigma$ has  a principal curvature of multiplicity $n-1.$
Hence, $\Sigma$ is conformally flat,
since $I\times_\omega \Q_\epsilon^n$ is conformally flat (cf. \cite[Examples 16.2]{dajczer-tojeiro}).
Being  Einstein and conformally flat,
$\Sigma$ has constant sectional curvature, and so is trivial.
This shows that (ii) implies (i) and  finishes the proof of the theorem.
\end{proof}

\section{Einstein Graphs in $\wi$} \label{sec-graphs}

Consider an oriented isometric immersion
\[f:M_0^{n-1}\rightarrow\q,\]
of a Riemannian manifold $M_0$ into $\q.$
Suppose that there exists a neighborhood $\mathscr{U}$
of $M_0$ in $T^\perp M_0$  without focal points of $f,$ that is,
the restriction of the normal exponential map $\exp^\perp_{M_0}:T^\perp M_0\rightarrow M$ to
$\mathscr{U}$ is a diffeomorphism onto its image. Denoting by
$\eta$ the unit normal field  of $f,$   there is an open interval $I\owns 0$
such that, for all $p\in M_0,$ the curve
\begin{equation}\label{eq-geodesic}
\gamma_p(s)=\exp_{\scriptscriptstyle \q}(f(p),s\eta(p)), \, s\in I,
\end{equation}
is a well defined geodesic of $\q$ without conjugate points. Thus,
for all $s\in I,$
\[
\begin{array}{cccc}
f_s: & M_0 & \rightarrow & \q\\
     &  p       & \mapsto     & \gamma_p(s)
\end{array}
\]
is an immersion of $M_0$ into $\q,$ which is said to be \emph{parallel} to $f.$
Notice that, given $p\in M_0$, the tangent space $f_{s_*}(T_p M_0)$ of $f_s$
at $p$ is the parallel transport of $f_{*}(T_p M_0)$ along
$\gamma_p$ from $0$ to $s.$ We also remark that,  with the induced metric,
the unit normal  $\eta_s$  of $f_s$ at $p$ is given by
\[\eta_s(p)=\gamma_p'(s).\]

\begin{definition}
Let $\phi:I\rightarrow \phi(I)\subset\R$ be an increasing diffeomorphism, i.e., $\phi'>0.$
With the above notation, we call the set
\begin{equation}\label{eq-paralleldescription1}
\Sigma:=\{(\phi(s),f_s(p))\in \wi\,;\, p\in M_0, \, s\in I\},
\end{equation}
the \emph{graph} determined by $\phi$ and  $\{f_s\,;\, s\in I\}$ or $(\phi,f_s)$-\emph{graph}, for short.
\end{definition}

Let $\Sigma\subset I\times_\omega\Q_\epsilon^n$ be a $(\phi,f_s)$-graph
with the induced metric from $I\times_\omega\Q_\epsilon^n.$
For an arbitrary  point $x=(\phi(s), f_s(p))$ of $\Sigma,$ one has
\[T_x\Sigma=f_{s_*}(T_p M_0)\oplus {\rm Span}\,\{\partial_s\}, \,\,\, \partial_s=\eta_s+\phi'(s)\partial_t,\]
where $\eta_s$ denotes the unit normal field of $f_s$\,.

Writing, by abuse of notation,
$\omega\circ\phi=\omega,$  a  unit normal  to $\Sigma$ is
\begin{equation} \label{eq-normal}
N=\frac{-\phi'/\omega}{\sqrt{1+(\phi' / \omega)^2}}\dfrac{\eta_s}{\omega}+\frac{1}{\sqrt{1+(\phi' / \omega)^2}}\partial_t.
\end{equation}
In particular, the angle function of $\Sigma$ is
\begin{equation} \label{eq-thetaparallel}
\Theta=\frac{1}{\sqrt{1+(\phi' / \omega)^2}}\,\cdot
\end{equation}
So, if we set
\begin{equation}\label{eq-rho}
\rho:=\frac{\phi' / \omega}{\sqrt{1+(\phi' / \omega)^2}},
\end{equation}
we can write $N$ as
\begin{equation}\label{eq-normal2}
N=-\rho\dfrac{\eta_s}{\omega}+\theta\partial_t\,,
\end{equation}
which yields $1=\rho^2+\theta^2,$ i.e., $\rho=\|T\|$. It also follows from \eqref{eq-normal2} that
\[
\overbar\nabla_{\partial s}N=
-\left(\dfrac{\rho}{\omega}\right)'\eta_s-
\dfrac{\rho}{\omega}\overbar\nabla_{\partial s}\eta_s+\langle \nabla\theta,\partial_s\rangle\partial_t+\theta\overbar\nabla_{\partial_s}\partial_t\,.
\]
However, by the identities \eqref{eq-connectionwarped}, setting $\zeta=\omega'/\omega$ and, again by abuse of notation, $\zeta\circ\phi=\zeta,$ we have that
$$
\begin{array}{rcl}
\overbar\nabla_{\partial s}\eta_s &=& \overbar{\nabla}_{\eta_s} \eta_s + \phi' \overbar{\nabla}_{\partial_t} \eta_s \\
&=& \widetilde\nabla_{\eta_s}\eta_s - \zeta \langle \eta_s, \eta_s \rangle \partial_t + \phi' \zeta \eta_s \\
&=& - \zeta \omega^2 \partial_t + \phi' \zeta \eta_s.
\end{array}
$$
and  $\overbar\nabla_{\partial_s}\partial_t=\overbar\nabla_{\eta_s}\partial_t=\zeta\eta_s$\,. Therefore,
\[
A\partial_s=-\overbar\nabla_{\partial s}N=-(\rho\zeta\omega+\langle\nabla\theta,\partial_s\rangle)\partial_t+\left(\dfrac{\rho'}{\omega}-\theta\zeta\right)\eta_s\,,
\]
which implies that $A\partial_s$ is orthogonal to any vertical field $X\in\{\partial_s\}^\perp\subset T\Sigma$. Thus,
$\partial_s$ is a principal direction of $\Sigma$ with eigenvalue
\[
\lambda_1=\frac{\rho'}{\omega}-\theta\zeta.
\]
In particular, $\lambda_1$ is a function of $s$ alone.

From \eqref{eq-thetaparallel}, we have that $0<\theta <1.$
Thus, $T$ is non vanishing on $\Sigma$ and parallel to $\partial_s$\,,
which implies that $T$ is also a principal direction of
$\Sigma$ with eigenvalue $\lambda_1.$

Now, let $X_2\,, \dots ,X_n$ be orthonormal vertical principal directions of $\Sigma.$ By Lemma \ref{lem-verticalsection},
these fields  are also principal directions of the vertical sections $\Sigma_s=\Sigma_{\phi(s)}$
with  corresponding principal curvatures
\begin{equation} \label{eq-lambdais}
\lambda_i^s=-\frac{\omega}{\|T\|}(\lambda_i+\theta\zeta)=-\frac{\omega}{\rho}(\lambda_i+\theta\zeta), \,\, \,\, i=2,\dots ,n.
\end{equation}

Summarizing,
we have the following result.
\begin{proposition} \label{prop-graph}
Let $\Sigma$ be a $(\phi,f_s)$-graph in $\wi$ with unit normal
$N$ (with respect to the induced metric) as in \eqref{eq-normal}.
Then, $\Sigma$ has the $T$-property, and
its principal curvatures are given by
\begin{equation}
\lambda_1=\dfrac{\rho'(s)}{\omega(\phi(s))}-\theta\zeta(\phi(s)), \quad
\lambda_i=-\dfrac{\rho(s)}{\omega(\phi(s))}\lambda_i^s(p)-\theta\zeta(\phi(s)), \,\, \,\, i=2,\dots ,n, \label{lambdas-sum}
\end{equation}
where $\theta$ is as in \eqref{eq-thetaparallel}, $\zeta=(\omega'\circ\phi)/(\omega\circ\phi),$
and $\lambda_i^s(p)$ is the $i$-th principal curvature of the parallel $f_s:M_0\rightarrow\q$ at $p\in M_0.$
\end{proposition}

It is well-known that the principal curvatures $\lambda_i^s(p)$ in the  statement
of Proposition \ref{prop-graph} are given by (see \cite{cecil, dominguez-vazquez}):
\begin{equation}\label{eq-lambdais007}
\lambda_i^s(p) = \dfrac{\epsilon S_{\epsilon}(s) + C_{\epsilon}(s)\lambda_i^0(p)}{C_{\epsilon}(s)-S_{\epsilon}(s)\lambda_i^0(p)},
\end{equation}
where $\lambda_i^0$ is the $i$-th principal curvature of $f=f_0,$ and
$C_\epsilon, S_\epsilon$ are as in Table \ref{table-trigfunctions}.
\begin{table}[thb]%
\centering %
\begin{tabular}{cccc}
\toprule %
   {{\small\rm Function}}                & $\epsilon=0$ & $\epsilon=1$   & $\epsilon=-1$ \\\otoprule %
$C_\epsilon (s)$    & $1$          & $\cos s$        & $\cosh s$     \\\midrule
$S_\epsilon (s)$    & $s$          & $\sin s$         & $\sinh s$   \\\bottomrule
\end{tabular}
\vtt
\caption{Definition of $C_\epsilon$ and $S_\epsilon$}
\label{table-trigfunctions}
\end{table}

In particular, the following elementary equalities hold:
\begin{equation}\label{relations-c-s}
  \begin{aligned}
   C_{\epsilon}'(s)&=-\epsilon S_{\epsilon}(s)\\
    S_{\epsilon}'(s)&=C_{\epsilon}(s)\\
    C_{\epsilon}^2(s)+\epsilon S^2_{\epsilon}(s) &= 1.
  \end{aligned}
\end{equation}

From identities \eqref{lambdas-sum}, \eqref{eq-lambdais007} and \eqref{relations-c-s}, we have
\begin{eqnarray}
\lambda_1 = - \dfrac{\theta}{\phi'} \dfrac{d}{ds} \log (\theta \omega), \label{lemma-lambda1} \\
\lambda_i^s(p) = - \dfrac{d}{ds} \log \left(C_{\epsilon}(s) - S_{\epsilon}(s) \lambda_i^0(p)\right), \label{lambda-i-int} \\
\dfrac{d}{ds} \lambda_i^s(p) = \epsilon + (\lambda_i^s(p))^2. \label{lambda-i-dev}
\end{eqnarray}

Let $f:M_0\rightarrow\q$ be a hypersurface, and let
$$\mathscr F=\{f_s:M_0\rightarrow\q\,;\, s\in I\owns 0\}$$
be a family of parallel hypersurfaces to $f=f_0.$ We say that $f$ is \emph{isoparametric} if any hypersurface $f_s$ has constant mean curvature.
It was proved by Cartan that  $f$ is isoparametric if and only if any $f_s$ has constant principal curvatures
$\lambda_i^s$ (possibly depending on $i$ and $s$).

Suppose that, for an integer $d>1,$  $\lambda_1, \dots ,\lambda_{d}$
are  pairwise distinct principal curvatures of an isoparametric
hypersurface $f:M_0\rightarrow\q.$ The following formula,
also due to Cartan, holds  for any fixed $i\in\{1,\dots, d\}$ (cf. \cite{cecil, dominguez-vazquez}):
\[
\sum_{i\ne j}n_j\frac{\epsilon+\lambda_i\lambda_j}{\lambda_i-\lambda_j}=0,
\]
where $n_j$ denotes the multiplicity of $\lambda_j.$

If $n>2$ and $f$ has only two distinct principal curvatures, say $\lambda$ and $\mu,$ we have
from equality \eqref{eq-lambdais007} that  the same is true for each parallel $f_s.$ Denoting their
corresponding principal curvatures by $\lambda^s$ and $\mu^s,$
Cartan's formula yields
\begin{equation}\label{eq-cartan}
  \lambda^s\mu^s=-\epsilon \,\,\, \forall s\in I.
\end{equation}

\section{Proofs of Theorems \ref{L1-grapheinstein}--\ref{th-final}} \label{sec-last}

\begin{proof}[Proof of Theorem \ref{L1-grapheinstein}]
By Lemma \ref{lem-T}, $\Sigma$ has the $T$-property. So, the same is true for
$\varphi(\Sigma)\subset J\times\Qe,$
where $\varphi$ is the map given in  \eqref{eq-varphi}. Also, since $\varphi$ is a conformal diffeomorphism,
the angle function of $\varphi(\Sigma)$ is non vanishing in $J\times\Qe.$
Hence, by \cite[Theorem 1]{tojeiro}, $\varphi(\Sigma)$ is given locally by a
$(\tilde\phi,f_s)$-graph, which implies that $\Sigma$ is locally a
$(\phi,f_s)$-graph,
where $\phi=G^{-1}\circ\tilde\phi$, and $G$ is the diffeomorphism
$G:I\rightarrow J$ such that $G'=1/\omega.$ This proves (i).
(By abuse of notation, we shall  call $\Sigma$ such a local $(\phi,f_s)$-graph.)

Let  $\{X_1\,,\dots ,X_n\}$ be an orthonormal basis of principal directions of
$\Sigma$ such that $X_1=T/\|T\|.$
Since $\rho=||T||$, equations \eqref{T-einstein-lambda1} become
\begin{equation}\label{T-einstein-lambda007}
  \begin{aligned}
    \lambda_1^2 - H \lambda_1 + (n-1)(\beta \rho^2 + \alpha) + \Lambda &= 0\\[1ex]
    \lambda_i^2 - H \lambda_i + \beta \rho^2 + (n-1) \alpha + \Lambda &=0,
  \end{aligned}
\end{equation}
where, by abuse the notation, we
write $\alpha\circ\xi=\alpha$ and $\beta\circ\xi=\beta.$

The second of the above equations  gives that $\Sigma$ has at most two distinct
principal curvatures among the last $n-1$. Let us call them $\lambda$ and $\mu$.
So, one has
\begin{equation}\label{system-lambda-mu}
  \begin{aligned}
 \lambda + \mu &= H  \\[1ex]
\lambda \mu &= (n-1)(\alpha + \beta \rho^2) + \Lambda.
\end{aligned}
\end{equation}

Let us assume first that $\lambda_1$
vanishes in an open interval $I_0 \subset I$. Denoting the multiplicity of
$\lambda$ and $\mu$ by $n_{\lambda}$ and $n_{\mu}$, respectively,
it follows from the first equation in \eqref{T-einstein-lambda007}
that equations \eqref{system-lambda-mu} can be  rewritten as
\begin{equation}
\begin{array}{rcl}
(1-n_{\lambda}) \lambda + (1-n_{\mu}) \mu &=& 0  \\
\lambda \mu &=& -(n-2) \beta \rho^2. \label{system-lambda-mu-rewritten}
\end{array}
\end{equation}

Since $n>2$ and $\beta ||T||$ never vanishes,
it follows from equations \eqref{system-lambda-mu-rewritten}
that  $\beta>0$, and that $\lambda$ and $\mu$ depend only on $s.$
Since $\lambda_1$ also depends only on $s,$
we conclude that $f_s$ is isoparametric.

By continuity, we can assume now that  $\lambda_1$ never vanishes.
Since $\lambda_1$ and $\beta \rho^2 + \alpha$ depend only on $s,$ it follows
from the first equality in \eqref{T-einstein-lambda007} that
the same is true for $H.$ This,
together with \eqref{eq-lambdais}, gives that the mean curvature $H_s$
of $f_s$ is a function of $s$ alone, which implies that the family
$\{f_s\}$ is isoparametric. In addition,  identity \eqref{eq-lambdais007} gives that
$f_s,$ just as $f=f_0,$ has at most two distinct principal curvatures, which we shall
denote by $\lambda^s$ and $\mu^s.$ This finishes the proof of (ii).

To prove (iii), let us suppose by contradiction that $\lambda\neq\mu$ and, consequently,
that $\lambda^s\neq\mu^s$. Since we are assuming $\lambda_1\equiv0$, it follows from \eqref{lambdas-sum} that
\begin{equation}
\begin{array}{rcl}
\theta \zeta &=& \dfrac{\rho'}{\omega} \\[1ex]
\lambda &=& - \dfrac{\rho \lambda^s+\rho'}{\omega} \\[1ex]
\mu &=& - \dfrac{\rho \mu^s+\rho'}{\omega}\,\cdot
\end{array} \label{lambda-mu-particular}
\end{equation}
In particular, equations \eqref{system-lambda-mu-rewritten} can be written as
\begin{equation} \label{system-lambda-mu-v2}
\begin{aligned}
(n_{\lambda}-1) \lambda^s + (n_{\mu}-1) \mu^s + (n-3) \dfrac{\rho'}{\rho} &= 0  \\[1ex]
\left( \dfrac{\rho'}{\rho} \right)^2 + (\lambda^s+\mu^s) \dfrac{\rho'}{\rho} + (n-3)\epsilon + (n-2) \omega^2 \zeta' &= 0.
\end{aligned}
\end{equation}

Now, equation \eqref{lemma-lambda1} yields  $\theta\omega=\sqrt{c}$
for some positive constant $c$. Therefore, from the first equation in \eqref{lambda-mu-particular},
we have $\zeta={\rho'}/{\sqrt{c}}$. Moreover, from the expressions of $\Theta$ and $\rho$
in \eqref{eq-thetaparallel} and \eqref{eq-rho},
one has ${\rho}/{\theta} = {\phi'}/{\omega}$. Hence,
\[
\omega^2\zeta'=\frac{\omega^2 \rho''}{\phi'\sqrt{c}}=\frac{\omega \rho''}{\phi'\Theta}=\frac{\rho''}{\rho}\,,
\]
so that
\[
\omega^2 \zeta' = \left( \dfrac{\rho'}{\rho} \right)'+\left( \dfrac{\rho'}{\rho} \right)^2.
\]
From this and the second identity in \eqref{system-lambda-mu-v2}, we have
\begin{equation} \label{system-lambdalmu-v22}
(n-1) \left( \dfrac{\rho'}{\rho} \right)^2 + (\lambda^s+\mu^s) \dfrac{\rho'}{\rho} + (n-3)\epsilon + (n-2) \left( \dfrac{\rho'}{\rho} \right)' = 0.
\end{equation}
Using \eqref{lambda-i-dev} and the first equality in \eqref{system-lambda-mu-v2}, we also have
\begin{equation}
\left( \dfrac{\rho'}{\rho} \right)' = -\epsilon + \left(\dfrac{1-n_{\lambda}}{n-3} \right) (\lambda^s)^2 + \left(\dfrac{1-n_{\mu}}{n-3} \right) (\mu^s)^2. \label{rho-derivative}
\end{equation}

Finally, Cartan's formula  $\lambda^s \mu^s = -\epsilon$ and equalities
\eqref{system-lambda-mu-v2}--\eqref{rho-derivative} yield
$$
\dfrac{(n-1)(1-n_{\lambda})(1-n_{\mu})}{(n-3)^2} \left( (\lambda^s)^2 + (\mu^s)^2+2 \epsilon \right) = 0.
$$

Since, from \eqref{system-lambda-mu-rewritten}, we have $n_{\lambda} > 1$
and $n_{\mu}>1,$ the above equation  reduces to $(\lambda^s)^2 + (\mu^s)^2+2 \epsilon = 0.$
Again from $\lambda^s\mu^s=-\epsilon,$ we have $(\lambda^s-\mu^s)^2 = 0,$
which yields $\lambda^s\equiv\mu^s,$ so that $\lambda\equiv\mu.$
Therefore, all vertical sections of $\Sigma$ are totally umbilical. So,
Theorem \ref{th-einstein} applies and gives that $\Sigma$ is trivial.
This shows (iii) and finishes the proof of the theorem.
\end{proof}

\begin{proof}[Proof of Theorem \ref{Einstein-cylinder}]

Since $\theta<1$ is constant, we have that $T$ never
vanishes on $\Sigma$ and, from \eqref{eq-gradthetawarp},
that $T$ is a principal direction with associated principal curvature
\begin{equation}\label{eq-lambda1Thetacte}
\lambda_1 = -\theta\dfrac{\omega'}{\omega}\,\cdot
\end{equation}

Let us suppose that $\beta\circ\xi$ vanishes in a connected open set
$\Sigma'\subset\Sigma.$ Then, from Theorem \ref{th-betazero},
$\alpha\circ\xi$ is constant on $\Sigma'$, and  $\Sigma'$ has constant sectional curvature,
unless $\sigma:=\Lambda+(n-1)(\alpha\circ\xi)$ is negative. In this case,
Theorem \ref{th-betazero} also gives that all
principal curvatures of $\Sigma'$ are constant. Hence, from \eqref{eq-lambda1Thetacte},
$\omega'/\omega$ is constant on $\Sigma'.$ Since $\alpha\circ\xi$ is also constant, we have
from \eqref{eq-def-alpha-beta} that $\omega$ is constant on $\Sigma',$ which implies that
$\lambda_1$ vanishes there. However, from Lemma \ref{lem-T-einstein}, we have
\[
0=\lambda_1^2-H\lambda_1+\sigma=\sigma<0,
\]
which is clearly a contradiction.

From the above considerations and the connectedness of $\Sigma,$ we can assume
that $\beta\circ\xi$ never vanishes on $\Sigma,$ which makes it an ideal Einstein hypersurface.
Assuming also that $\theta>0,$ we have from Theorem \ref{L1-grapheinstein} that $\Sigma$ is locally a $(\phi, f_s)$-graph, where
$\{f_s\}$ is isoparametric.
Since $\theta$ is constant on $\Sigma,$ so is the $\rho$-function of
any such graph,
for $\theta^2+\rho^2=1.$ Thus,
from Theorem \ref{L1-grapheinstein}-(iii), $\Sigma$ is trivial.
This proves (i).


Let us suppose now that $\theta$ vanishes on $\Sigma.$ Then, from \eqref{eq-lambda1Thetacte},
$\lambda_1 \equiv 0$. In this case, we have from  \eqref{T-einstein-lambda1} that
\begin{equation} \label{einstein-theta1}
\begin{aligned}
(n-1)(\beta + \alpha) + \Lambda &= 0 \\[1ex]
\lambda_i^2 - H \lambda_i + \beta + (n-1) \alpha + \Lambda &=0,
\end{aligned}
\end{equation}
where, by abuse of notation, we write $\beta\circ\xi=\beta$ and $\alpha\circ\xi=\alpha.$

Considering the first of the above equations and
the definitions of $\alpha$ and $\beta$ given in \eqref{eq-def-alpha-beta},
we conclude that $\omega$ must satisfy the following ODE
$$
\dfrac{\omega''}{\omega} = -\dfrac{\Lambda}{n-1}\,\cdot
$$
A first integration then gives, for some constant $c,$
\begin{equation}
(\omega')^2+\dfrac{\Lambda}{n-1}\omega^2+c=0, \label{eq-edo-omega2}
\end{equation}
which, together with \eqref{eq-def-alpha-beta}, yields
\begin{equation} \label{eq-alpha-beta-omega010}
\alpha =-\dfrac{\Lambda}{n-1}-\dfrac{c+\epsilon}{\omega^2} \quad \textnormal{and} \quad  \beta=\dfrac{c+\epsilon}{\omega^2}\,\cdot
\end{equation}

If $c+\epsilon=0$, then $\beta=0$ and we are
in the position of Theorem \ref{th-betazero}-(ii), in which case
$\Sigma$ has constant sectional curvature
and its shape operator has rank at most $1.$

If   $c+\epsilon\neq 0,$ we have from \eqref{eq-alpha-beta-omega010} that
$\beta$ does not vanish identically on $I,$ so that $I\times_\omega\q$ has
nonconstant sectional curvature.
Also, since $\theta \equiv 0$, we have $T=\partial_t,$ which gives that $\Sigma$ is
a cylinder over a hypersurface $\Sigma_0$ of $\mathbb{Q}^n_{\epsilon}.$
Denoting by $\lambda_i^0$ the $i$-th
principal curvature of $\Sigma_0,$ we have from
equality \eqref{eq-principalcurvatureslemma}
in Lemma \ref{lem-verticalsection} that
$\lambda_i^0 = -\omega(t) \lambda_i.$ In this case,
the second equation in \eqref{einstein-theta1} becomes
\begin{equation} \label{eq-sigma0}
(\lambda_i^0)^2 -  H^0 \lambda_i^0 -(n-2)(c+\epsilon)=0,
\end{equation}
where $H^0$ is the mean curvature of $\Sigma_0.$
Therefore, at each point of $\Sigma_0,$
we have at most two distinct principal curvatures. Let us call them $\lambda^0$ and $\mu^0$.

Suppose $c+ \epsilon <0$. If $\lambda^0 \neq \mu^0$ at some point, then
\begin{equation}\label{eq-lambdamuzero}
\begin{array}{rcl}
(k-1) \lambda^0 + (n-k-2) \mu^0 &=& 0  \\
\lambda^0 \mu^0 &=& -(n-2) (c+\epsilon).
\end{array}
\end{equation}
Then $\lambda_0$ and $\mu_0$ do not vanish and have the same sign. The first equation of \eqref{eq-lambdamuzero}
implies that  $k-1=n-k-2=0$ and, therefore, $n=3$. Since we are supposing $n>3,$
we conclude that all points of $\Sigma_0$ are umbilical, which implies
that  $\Sigma$ has constant sectional curvature.

If $c + \epsilon>0,$
the roots of \eqref{eq-sigma0} are
necessarily distinct,  which implies that $\lambda^0 \neq \mu^0$ everywhere on $\Sigma_0.$
Moreover, from \eqref{eq-lambdamuzero},
$\lambda^0$ and $\mu^0$ are constant functions.
In this case, Cartan's identity $\lambda^0 \mu^0 = - \epsilon$ implies that
$\epsilon = 1$. Otherwise, again by the first equation of \eqref{eq-lambdamuzero},
we would have $n=3.$

Therefore, denoting by $\s^k(c)$ the $k$-dimensional sphere of constant sectional curvature $c,$ it follows from
\cite[Theorems 2.5,\,3.4]{ryan} that, for some $k\in\{1,\dots,n-2\},$
$\Sigma_0$ is congruent to the standard  Riemannian product $\s^k(c_1)\times\s^{n-1-k}(c_2),$ where
\begin{equation}\label{radi-cylinder}
c_1={\dfrac{n-3}{k-1}} \quad\textnormal{and}\quad c_2={\dfrac{n-3}{n-k-2}}\,\cdot
\end{equation}
This finishes our proof.
\end{proof}

\begin{remark} \label{rem-thetacte}
It follows from the proof of Theorem \ref{Einstein-cylinder} that, except for the case (ii)-(a),
the Einstein hypersurface $\Sigma$ in its statement is necessarily ideal.
\end{remark}

In the next proof, we shall consider the fact that
the isoparametric hypersurfaces of $\q$ having at most two distinct
principal curvatures are completely classified. They are
the totally umbilical hypersurfaces, in the case of a single principal curvature,
and the tubes around totally geodesic submanifolds,
in the case of two distinct principal curvatures. In Euclidian space, for instance,
they constitute hyperplanes, geodesic spheres, and generalized cylinders
$\s^{k}\times\R^{n-k-1},$ $k\in\{1,\dots, n-2\}$ (see \cite{cecil,dominguez-vazquez} for more details).

\begin{proof}[Proof of Theorem \ref{th-final}]
Keeping the notation of the proof of Theorem \ref{th-einstein}, we have that $\Sigma$ has at most
three distinct principal curvatures, $\lambda_1, \lambda$ and $\mu.$ Since $(\beta\circ\xi)T$ never vanishes on
$\Sigma,$ the principal curvature $\lambda_1$ cannot satisfy both equations in Lemma \ref{lem-T-einstein}. Thus,
$\lambda\ne\lambda_1\ne\mu,$ i.e., $\lambda_1$ has multiplicity one, so that
$\Sigma$ has at least two distinct principal curvatures.
Let us define the set
\[
C:=\{x\in\Sigma\,;\, \lambda(x)=\mu(x)\}.
\]

If $\Sigma=C,$ we are done. Analogously, if
$C=\emptyset,$ we have that
$\Sigma$ has three distinct principal curvatures everywhere. Also, since
$\{T\}^\perp$ is invariant by the shape operator $A$ of $\Sigma,$
and $\lambda$ and $\mu$ are the (distinct) eigenvalues of $A|_{\{T\}^\perp},$
we have from \cite[Proposition 2.2]{ryan}  that both $\lambda$ and $\mu$
have constant multiplicity.

Now, suppose that $C$ and $\Sigma-C$ are both nonempty and let
$x$ be a boundary point of $\Sigma-C.$ Since $C$ is closed, one has
$x\in C,$ so that the vertical section $\Sigma_t\owns x$ is umbilical at $x$
(by \eqref{eq-principalcurvatureslemma} in Lemma \ref{lem-verticalsection}).
Furthermore, there is no open set $\Sigma'$ of $\Sigma$ containing $x$
such that $\theta$ vanishes on $\Sigma'.$ Indeed, assuming otherwise,
we have that $\Sigma'$ correspond to
one of the cases (b) or (c)
of Theorem \ref{Einstein-cylinder}-(ii) (see Remark \ref{rem-thetacte}).
However, in any of them,
$\lambda=\mu$ or $\lambda\ne\mu$
everywhere on $\Sigma',$ which contradicts the fact that
$x$ is a boundary point of $\Sigma-C.$

It follows from the above considerations that, in any open set $\Sigma'$ of
$\Sigma$ containing $x,$  there exists
$y\in(\Sigma-C)\cap\Sigma'$ such that $\theta(y)\ne0.$
Then, from Theorem \ref{L1-grapheinstein}, there exists a local $(\phi,f_s)$-graph
$\Sigma(y)\owns y$ in $\Sigma-C$
whose parallel family $\{f_s\}$ is isoparametric in $\q.$ In this case,
each parallel $f_s$ has precisely  two distinct principal curvatures everywhere,
so that they are tubes around totally  geodesic submanifolds.

Now, we can choose  $y$ as above in  such a way that
$x$ is a boundary point of $\Sigma(y).$
Then, denoting by $\pi$ the vertical projection of $I\times\q$ over $\q,$
one has that an open neighborhood
$U\owns\pi(x)$ in $\pi(\Sigma_t)\subset\q$ is the limit set of a
family of parallel open subsets of tubes of the family $\{f_s\}.$ In this case,
it is easily seen that $U$ itself must be part of a
tube of $\{f_s\}$. However, $U$ is umbilical at $\pi(x),$ and tubes have no umbilical points.
This contradiction shows that either $C$ or $\Sigma-C$  is empty, proving the
first part of the theorem. As for the last part,  just apply  Theorem \ref{th-einstein}.
\end{proof}

\end{document}